\documentclass[a4paper,11pt]{article}
\usepackage{amsmath}
\usepackage{amsthm}
\usepackage{amsfonts}
\usepackage{amssymb}
\usepackage{indentfirst}
\usepackage[left=2.5cm,top=2.5cm,right=2.5cm,bottom=2.5cm]{geometry}

\newtheorem{lem}{Lemma}[section]
\newtheorem{thm}{Theorem}[section]
\newtheorem{co}[thm]{Corollary}

\numberwithin{equation}{section}

\begin{document}
\title{Moments, moderate  and large deviations for a branching process
in a random environment}
\author{Chunmao HUANG$^{a,b}$, Quansheng LIU$^{a,b,}$\footnote{Corresponding author at:  LMAM, Universit\'e de Bretagne-Sud, Campus de Tohannic,
BP 573, 56017 Vannes, France. Tel.: +33 2 9701 7140; fax: +33 2 9701 7175.\newline \indent \ \ Email addresses: sasamao02@gmail.com (C. Huang), quansheng.liu@univ-ubs.fr (Q. Liu).}
\\
\small{\emph{$^{a}$LMAM, Universit\'e de Bretagne-Sud, Campus de Tohannic,
BP 573, 56017 Vannes, France}}\\\small{\emph{ $^{b}$Universit\'e Europ\'eenne de Bretagne, France}}}
\date{}
\maketitle

\textsc{Abstract}. Let $(Z_{n})$ be a supercritical branching
process in a random environment $\xi $, and $W$ be the limit of the normalized population size  $Z_{n}/\mathbb{E}[Z_{n}|\xi ]$. We
show large and moderate deviation principles for the sequence $\log Z_{n}$
(with appropriate normalization). For the proof, we calculate the critical
value for the existence of harmonic moments of $W$, and show an equivalence
for all the moments of $Z_{n}$.  Central limit theorems on $W-W_n$ and $\log Z_n$ are also established.
\\*


\\*

\section{Introduction and main results}

As an important extension of the Galton-Watson process, the model of  branching process in a random environment was introduced
first by Smith \& Wilkinson (1969, \cite{smith}) for the independent environment case, and then by Athreya \& Karlin (1971, \cite{a1})
for the stationary and ergodic environment case. See also Athreya \& Ney (1972, \cite{a}) and Tanny (1977, \cite{tanny1}; 1988, \cite{tanny2}) for some basic results
on the subject. The study of asymptotic properties of a branching process in a random environment has recently  received attention, see for example Afanasyave, Geiger, Kersting \& Vatutin (2005, \cite{af1} \& \cite{af2}),  Kozlov (2006, \cite{kozlov}), Bansaye \& Berestycki (2009, \cite{ba}),  Bansaye \& B\"oinghoff (2010, \cite{ba2}),
B\"oinghoff \& Kersting (2010, \cite{bo}), and B\"oinghoff,  Dyakonova, Kersting \& Vatutin (2010, \cite{bd}),  among others.
Here,  for a supercritical branching process $(Z_n)$  in a random environment, we shall mainly show asymptotic properties of the moments of $Z_n$, and prove moderate and large deviation principles for ($\log Z_n$). In particular, our result on the annealed harmonic moments completes that of Hambly (1992, \cite{ham}) on the quenched harmonic moments, and  extends the corresponding theorem of
Ney \& Vidyashanker (2003, \cite{ney}) for the Galton-Watson process; our moderate and large deviation principles complete the results of
  Kozlov (2006, \cite{kozlov}),  Bansaye \& Berestycki (2009, \cite{ba}),  Bansaye \& B\"oinghoff (2010, \cite{ba2}) and  B\"oinghoff \& Kersting (2010, \cite{bo}) on large deviations.

 Let us  give a description of the model.
Let $\xi =(\xi _{0},\xi _{1},\xi _{2},\cdots )$ be a
sequence of independent and identically distributed (i.i.d.) random variables taking values in some space
$\Theta ,$ whose realization determines a sequence of probability
generating functions
 \begin{equation}
f_{n}(s)=f_{\xi _{n}}(s)=\sum_{i=0}^{\infty }p_{i}(\xi
_{n})s^{i},\;s\in \lbrack 0,1],\qquad p_{i}(\xi _{n})\geq 0,\qquad
\sum_{i=0}^{\infty }p_{i}(\xi _{n})=1.
\end{equation}
A branching process $(Z_{n})_{n\geq 0}$ in the random
environment $ \xi $ can be defined as follows:
\begin{equation}
Z_{0}=1,\qquad Z_{n+1}=\sum_{i=1}^{Z_{n}}X_{n,i}\;\;n\geq 0,
\end{equation}
where  given the environment
$\xi $,  $X_{n,i}$ $(i=1,2,...)$ are independent of each other and independent of $Z_n$, and have
the same distribution determined by $f_{n}$.

Let $(\Gamma ,\mathbb{P}_{\xi })$ be the probability space under which the
process is defined when the environment $\xi $ is given. As usual,
$\mathbb{P}_{\xi }$ is called \emph{quenched law}. The total probability
space can be formulated as the product space $(\Gamma\times \Theta
^{{\mathbb{N}}} ,\mathbb{P})$, where $\mathbb{P}=\mathbb{P}_{\xi }\otimes \tau $
in the sense that for all measurable and positive function $g $, we
have
\begin{equation*}
\int gd\mathbb{P}=\int \int g(\xi ,y)d\mathbb{P}_{\xi }(y)d\tau (\xi ),
\end{equation*}
where $\tau $ is the law of the environment $\xi $. The total probability $\mathbb{P}$ is usually called
\emph{annealed law}. The quenched law $\mathbb{P}_{\xi }$ may be considered
to be the conditional probability of the annealed law $\mathbb{P}$ given $\xi
$. The expectation with respect to $\mathbb{P}_{\xi }$ (resp. $\mathbb{P}$) will be denoted $\mathbb{E}_{\xi }$ (resp. $\mathbb{E}$).

For $\xi=(\xi_0,\xi_1, \cdots )$ and  $n\geq 0$, define

\begin{equation}
\label{mnp}
m_{n}(p) \; =\; m_n(p, \xi) \; =\; \sum_{i=0}^{\infty }i^p p_{i}(\xi _{n}) \quad \mbox{ for } p>0,
\end{equation}
\begin{equation}
\label{mn}
m_n = m_n (1), \;\;  \Pi_{0}=1\;\;\text{and} \;\;\Pi_{n}=m_{0}\cdots m_{n-1}\;\text{ for }n\geq 1.
\end{equation}
Then  $m_n(p) =\mathbb{E}_\xi X_{n,i}^p$ and $\Pi_n=\mathbb{E}_\xi Z_n$.
 It is well known that the normalized population size
\begin{equation*}
W_{n}=\frac{Z_{n}}{\Pi_{n}}
\end{equation*}
is a nonnegative martingale under $\mathbb{P}_{\xi }$ (for each $\xi$) with respect to the filtration $\mathcal{F}_n=\sigma(\xi, X_{k,i}, 0\leq k\leq n-1, i=1,2,\cdots)$,
so that the limit
\begin{equation*}
W=\lim_{n\rightarrow \infty }W_{n}\;
\end{equation*}
exists almost sure (a.s.) with $\mathbb{E}W\leq 1$. We shall always assume that
\begin{equation}\label{MDE1.0}
\mathbb{E}\log m_{0}\in (0,\infty )\quad \text{ and }\quad
\mathbb{E}\frac{Z_{1}}{m_{0}}\log^{+}Z_{1}<\infty.
\end{equation}
The first condition means that the process is supercritical; the second implies that $W$ is non-degenerate. Hence (see e.g. Athreya \& Karlin (1971, \cite{a1}))
\begin{equation*}
\mathbb{P}_{\xi }(W>0)=\mathbb{P}_{\xi }(Z_{n}\rightarrow \infty )=\lim_{n\rightarrow
\infty }\mathbb{P}_{\xi }(Z_{n}>0)\quad a.s..
\end{equation*}
For simplicity, \underline{we write often  $ p_i$ for $p_i(\xi_0) $ and assume always}
$$
p_0=0\quad a.s.$$
Therefore $W>0$ and $Z_{n}\rightarrow \infty $ a.s..

It is known that $\frac{\log Z_{n}}{n}\rightarrow \mathbb{E}\log
m_{0}$ a.s. on $\{Z_{n}\rightarrow \infty \}$ (see e.g. Tanny (1977, \cite{tanny1})). We are interested
in the asymptotic properties of the corresponding deviation
probabilities. Notice that
\begin{equation}\label{MDE0}
\log Z_n =\log \Pi_n + \log W_n .
\end{equation}
Since $W_n\rightarrow W>0$ a.s.,  certain asymptotic properties of $\log
Z_n$ would be determined by those of $\log \Pi_n$.  We shall show that $\log Z_n$ and $\log \Pi_n$ satisfy the same limit theorems under suitable moment conditions.
\\*

At first, we  present a large deviation principle. Let $\Lambda (t )=\log
\mathbb{E}m_{0}^{t }$. Assume that $m_0$ is not a constant a.s. and  that $\Lambda (t )<\infty $
for all $t \in \mathbb{R}$. Let
\begin{equation*}
\Lambda ^{\ast }(x)=\sup_{t\in \mathbb{R}}\{t x-\Lambda (t)\}
\end{equation*}
be the Fenchel-Legendre transform of $\Lambda $. It is well known
(\cite{z}, Lemma 2.2.5) that $\Lambda^* (\mathbb{E}\log m_{0})=0$, $\Lambda
^{\ast }(x)$ is strictly increasing for $x\geq \mathbb{E}\log m_{0}$ and  strictly decreasing for
$x\leq \mathbb{E}\log m_{0}$; moreover,
\begin{equation*}
\Lambda ^{\ast }(x)=\left\{\begin{array}{ll}
					tx-\Lambda (t )&\text{if $x=\Lambda^{\prime }(t )$ for some $t\in\mathbb{R}$,}\\
					\infty&\text{if $x\geq\Lambda^{\prime }(\infty )$ or $x\leq\Lambda^{\prime }(-\infty )$ }.
					\end{array}\right.
\end{equation*}
In fact, $\Lambda ^{\ast }$ is the rate function with which $\log\Pi_n$ satisfies a large deviation principle.
We introduce the following assumption:
\newline
\newline
\emph{(\textbf{H}) There exist constants $\delta>0$ and $A>A_{1}>1$ such that a.s.}
\begin{equation}
A_{1} \; \leq \;  m_0 \quad\text{and}\quad m_0({1+\delta}) \; \leq \;  A^{1+\delta},
\end{equation}
(recall that $m_0$ and $m_0(1+\delta)$ were defined in (\ref{mnp}) and (\ref{mn})). Notice that the second condition implies that
$m_0 \leq A$ a.s.

The theorem below shows that  $\log Z_n$ and $\log \Pi_n$ satisfy
the same large deviation principle.

\begin{thm}[Large deviation principle] \label{MDT5.2}
Assume (H). If $\mathbb{E}Z_1^s<\infty$  for all $s>1$ and $p_1=0$ a.s., then for any measurable subset $B$ of $ \mathbb{R}$,
  \begin{eqnarray*}
  -\inf_{x\in B^o}\Lambda^*(x)&\leq&\liminf_{n\rightarrow\infty}\frac{1}{n}\log \mathbb{P}\left(\frac{\log Z_n}{n}\in B\right)\\
  &\leq&\limsup_{n\rightarrow\infty}\frac{1}{n}\log \mathbb{P}\left(\frac{\log Z_n}{n}\in B\right)
  \leq  -\inf_{x\in \bar{B}}\Lambda^*(x),
  \end{eqnarray*}
  where $B^o$ denotes the interior of $B$, and $\bar B$ its closure.
\end{thm}

From Theorem \ref{MDT5.2}, we obtain immediately

\begin{co}\label{MDC1.1}
Assume (H). If $\mathbb{E}Z_1^s<\infty$  for all $s>1$ and $p_1=0$ a.s., then
  \begin{eqnarray*}
 \lim_{n\rightarrow\infty}\frac{1}{n}\log \mathbb{P}\left(\frac{\log Z_n}{n}\leq x\right)=-\Lambda^*(x)\quad\text{for $x< \mathbb{E}\log m_0,$}\\
  \lim_{n\rightarrow\infty}\frac{1}{n}\log \mathbb{P}\left(\frac{\log Z_n}{n}\geq x\right)=-\Lambda^*(x)\quad\text{for $x> \mathbb{E}\log m_0.$}
    \end{eqnarray*}
\end{co}

\noindent\textbf{Remark.}
This result was shown by  Bansaye \& Berestycki  (2009, \cite{ba}) when (H) holds with $\delta=1$.
If $\mathbb{P}(p_{1}>0)>0$, the rate function for
the lower deviation is no longer $\Lambda ^{\ast }(x)$: in this
case, Bansaye \& Berestycki \cite{ba} proved that under certain
hypothesis,
\begin{equation*}
\lim_{n\rightarrow \infty }\frac{1}{n}\log \mathbb{P}\left( \frac{\log
Z_{n}}{n}\leq x\right) =-\chi (x)\quad \text{for}\;x<\mathbb{E}\log
m_{0},
\end{equation*}
where $\chi (x)=\inf_{t\in \lbrack 0,1]}\{-t\log \mathbb{E}p_{1}+(1-t)\Lambda
^{\ast }(\frac{x}{1-t})\}$. Obviously, $\chi (x)\leq \Lambda ^{\ast
}(x)$.

For the upper deviation and for branching processes with  special offspring distributions,  more precise results can be found in
   Kozlov (2006, \cite{kozlov}),  B\"oinghoff \& Kersting (2010, \cite{bo}), and  Bansaye  \&  B\"oinghoff (2010, \cite{ba2}).
\\*

Notice that the Laplace transform of $\log Z_n$ is
    $$\mathbb{E}e^{t\log Z_n}=\mathbb{E}Z_n^t.$$
Therefore, Theorem \ref{MDT5.2} is a consequence of the G\"artner-Ellis theorem  (see e.g. \cite{z}) and Theorem \ref{MDT5.3} below.

\begin{thm}[Moments of $Z_n$]\label{MDT5.3}
Let $t\in \mathbb{R}$. Suppose that one of the following conditions is satisfied:
\begin{equation*}
	\begin{array}{l}
	\text{(i) $t\in(0,1]$ and $\mathbb{E}m_0^{t-1}Z_1\log^+Z_1<\infty$;}\\
	\text{(ii) $t>1$ and $\mathbb{E}Z_1^t<\infty$;}\\
	\text{(iii) $t<0$, $\mathbb{E}p_1<\mathbb{E}m_0^t$, $\|p_1\|_\infty:=\emph{esssup} \ p_1<1$ and (H) holds.}
	\end{array}
\end{equation*}
Then for some constant $C(t)\in(0,\infty)$,
$$\lim_{n\rightarrow\infty}\frac{\mathbb{E}Z_n^t}{\left(\mathbb{E}m_0^t\right)^n}=C(t).$$
\end{thm}

For $t<0$, Theorem \ref{MDT5.3} is an extension of a result of Ney \& Vidyashankar (2003, \cite{ney}) on the Galton-Watson process. Theorem \ref{MDT5.3} can also be used to study the  convergence rate in a central limit theorem for $W-W_n$ (see Theorem \ref{MDT6.1}).

A key step in the proof of Theorem \ref{MDT5.3} is the study of the
harmonic moments (moments of negative orders) of $W$, which is of interest of its own.  The
following result is our main result on this subject.

\begin{thm}[Harmonic moments of $W$]\label{MDT1.3}
Let $a>0$. Assume (H) and $\|p_1\|_\infty<1$. Then
\begin{equation*}
\mathbb{E}W^{-a}<\infty \qquad \text{if and only if}\qquad \mathbb{E}p_{1}m_{0}^{a}<1.
\end{equation*}
\end{thm}

Theorem \ref{MDT1.3}  reveals that under certain conditions, the
number $a_0$ satisfying $\mathbb{E}p_1m_0^{a_0}=1$ is the critical value for
the existence of the harmonic moments   $\mathbb{E}W^{-a}(a>0)$. More
precisely, we have

\begin{co}
Assume (H) and $\|p_1\|_\infty<1$. If $\mathbb{E}p_{1}m_{0}^{a_{0}}=1$, then $\mathbb{E}W^{-a}<\infty $
if $0<a<a_{0}$ and $\mathbb{E}W^{-a}=\infty $ if $a\geq a_{0}$.
\end{co}

\noindent \textbf{Remark.}  Hambly (1992, \cite{ham}) proved that under an assumption  similar to (H), the number $\alpha _{0}:=-\frac{\mathbb{E}\log
p_{1}}{\mathbb{E}\log m_{0}}$ is the critical value for the a.s. existence of
the quenched moments $\mathbb{E}_{\xi }W^{-a}(a>0)$: namely, $\mathbb{E}_{\xi
}W^{-a}<\infty$ a.s. if $a<\alpha _{0}$ and $\mathbb{E}_{\xi }W^{-a}=\infty$ a.s. if $a>\alpha _{0}$. Here we obtain the critical value for the existence of the annealed moments instead of the quenched ones. Notice
that by Jensen's inequality and the equation $\mathbb{E}p_{1}m_{0}^{a_{0}}=1$, we see the
natural relation that $a_{0}\leq \alpha _{0}$.
\\*

Now we  consider  moderate deviations. Let $(a_{n})$ be a sequence of positive numbers satisfying
\begin{equation}
\label{an}
\frac{a_{n}}{n}\rightarrow 0\quad \text{and}\quad
\frac{a_{n}}{\sqrt{n}} \rightarrow \infty \;\text{ as}\;n\rightarrow
\infty.
\end{equation}
Similar to the case of large deviation principle, $\log Z_n$ and $\log \Pi_n$ satisfy the same moderate deviation principle.

\begin{thm}[Moderate deviation principle]\label{MDT4.3}
 Assume (H) and write $\sigma^2=\emph{var}(\log m_0)\in(0,\infty)$. Then for any measurable subset $B$ of $ \mathbb{R}$,
  \begin{eqnarray*}
  -\inf_{x\in B^o}\frac{x^2}{2\sigma^2}&\leq&\liminf_{n\rightarrow\infty}\frac{n}{a_n^2}\log \mathbb{P}\left(\frac{\log Z_n-n\mathbb{E}\log m_0}{a_n}\in B\right)\\
  &\leq&\limsup_{n\rightarrow\infty} \frac{n}{a_n^2}\log \mathbb{P}\left(\frac{\log Z_n-n\mathbb{E}\log m_0}{a_n}\in B\right)
  \leq  -\inf_{x\in \bar{B}}\frac{x^2}{2\sigma^2} ,
  \end{eqnarray*}
    where $B^o$ denotes the interior of $B$, and $\bar B$ its closure.
\end{thm}

Here and throughout the paper, var$(\log m_0)$ denotes the variance of $\log m_0$.

As in the case of large deviation principle, the proof of Theorem \ref{MDT4.3} is based on the G\"artner-Ellis theorem.
\\*

As another application of  Theorem \ref{MDT5.3}, we shall  also establish a central limit theorem for $W-W_n$ with exponential  convergence rate.
 Let
 \begin{equation}
 \delta_\infty^2(\xi)=\sum_{n=0}^{\infty}\frac{1}{\Pi_n}\left(\frac{m_n(2)}{m_n^2}-1\right)
 \end{equation}
 (recall that $m_n(2)=\sum_{i=1}^{\infty}i^2p_i(\xi_n)$ by (\ref{mnp})).
 Then $\delta_\infty^2$ is the variance of $W$ under $\mathbb{P}_\xi$ (see e.g. Jagers (1974, \cite{jajer})) if the series converges.  As usual, we write $T^n\xi=(\xi_n, \xi_{n+1}, \cdots)$ if $\xi=(\xi_0, \xi_1,\cdots)$ and $n\geq 0$.

\begin{thm}[Central limit theorem on $W-W_n$]
\label{MDT6.1}
Assume (H) and $\|p_1\|_\infty<1$.
 If $\mathbb{E}p_1<\mathbb{E}m_0^{-\epsilon/2}$, $\emph{essinf}\ \frac{m_0(2)}{m_0^2}>1$  and $\mathbb{E}Z_1^{2+\epsilon}<\infty$
 for some $\epsilon\in(0,1]$, then for some constant $C>0$,
\begin{equation}\label{MDET5.3.1}
\sup_{x\in\mathbb{R}}\left|\mathbb{P}\left(\frac{\Pi_n(W-W_n)}{\sqrt{Z_n}\delta_\infty(T^n\xi)}\leq x\right)-\Phi(x)\right|
\leq C\left(\mathbb{E}m_0^{-\epsilon/2}\right)^n.
\end{equation}
\end{thm}
Notice that the condition  $\mathbb{E}p_1<\mathbb{E}m_0^{-\epsilon/2}$  is automatically satisfied when $\epsilon>0$ is small enough.

Theorem \ref{MDT6.1} shows that $W-W_n$ (with appropriate normalization) satisfies a central limit theorem with an exponential convergence rate; it improves a recent result of Wang, Gao \& Liu (2010, \cite{wang}). For Galton-Watson process, Theorem \ref{MDT6.1} improves the convergence rate of Heyde \& Brown (1971, \cite{heyde2}), and coincides with that of Ney \& Vidyashanker (2003, \cite{ney}).
\\*

Finally, as $\log \Pi_n$ satisfies
a central limit theorem, it is natural that  the same would hold
for $\log Z_n$. In fact we have

\begin{thm}[Central limit theorem on $\log Z_n$]
\label{MDT1.1}
Assume that  $\sigma^2=\emph{var}(\log m_0)\in(0,\infty)$. Then
\begin{equation}  \label{MDE1.2}
\lim_{n\rightarrow\infty}\mathbb{P}\left(\frac{\log Z_n-n\mathbb{E}\log
m_0}{\sqrt{n}\sigma} \leq x\right)=\Phi(x),
\end{equation}
where
$\Phi(x)=\frac{1}{\sqrt{2\pi}}\int_{-\infty}^{x}e^{-{u^2}/{2}}du$ is the standard normal distribution function.
\end{thm}

The rest of the paper is organized as follows.
In Section \ref{MDS3}, we consider the harmonic moments of $W$ and prove Theorem \ref{MDT1.3}.
Section \ref{MDS4} is devoted to the study of the moments of $Z_n$ of all orders (positive or negative) and the large deviations of $\log Z_n$, where
Theorems \ref{MDT5.2} and \ref{MDT5.3} are proved with additional informations.
In Section \ref{MDS5}, we consider the moderate deviations of $\log Z_n$ and prove Theorem \ref{MDT4.3}.
In Section \ref{MDS6},  we deal with central limit theorems and prove   Theorems \ref{MDT6.1} and \ref{MDT1.1}. We end the paper
by a short appendix showing a general result on large deviations.


\section{Harmonic moments of $W$}\label{MDS3}
In this section, we shall  study the harmonic
moments of $W$, i.e. $\mathbb{E}W^{-s}$ $(s>0)$,
which are closely  related to the corresponding moments of $W_{n}$.
The following lemma reveals their relations.

\begin{lem}
\label{MDL3.1}  Assume (\ref{MDE1.0}). Then for any convex function
$\varphi:\mathbb{R}_+\rightarrow \mathbb{R}_+$,
$$
\lim_{n\rightarrow\infty}\mathbb{E}_\xi\varphi (W_{n})=\sup_{n}\mathbb{E}_\xi\varphi (W_{n})=\mathbb{E}_\xi\varphi (W)\quad a.s.,
$$
and
$$
\lim_{n\rightarrow\infty}\mathbb{E}\varphi (W_{n})=\sup_{n} \mathbb{E}\varphi (W_{n})=\mathbb{E}\varphi (W).
$$
In particular, for all $s>0$,
$$\lim_{n\rightarrow\infty}\mathbb{E}_\xi W_{n}^{-s}=\sup_{n}\mathbb{E}_\xi W_{n}^{-s}=\mathbb{E}_\xi W^{-s}\quad a.s., $$
and
$$\lim_{n\rightarrow\infty}\mathbb{E }W_{n}^{-s}=\sup_{n}\mathbb{E}W_{n}^{-s}=\mathbb{E} W^{-s}. $$
\end{lem}

\begin{proof}[Proof]
Recall that by (\ref{MDE1.0}), $W_n\rightarrow W$ in $L^1$. Therefore, $W_n=\mathbb{E}(W|\mathcal
{F}_n)$ a.s.. By the conditional Jensen's inequality,
$$\mathbb{E}(\varphi(W)|\mathcal {F}_n)\geq\varphi(\mathbb{E}(W|\mathcal {F}_n))=\varphi(W_n)\qquad a.s.,$$
so $\mathbb{E}\varphi(W)\geq\sup_n\mathbb{E}\varphi(W_n)$.  The other side comes from
Fatou's lemma. The equality $$\lim_{n\rightarrow\infty}\mathbb{E} \varphi (W_{n})=\sup_{n}\mathbb{E} \varphi (W_{n})$$ is obvious by the monotonicity of $\mathbb{E }\varphi (W_{n})$.
For the quenched moments, it suffices to repeat the proof above with $\mathbb{E}_\xi$ in place of $\mathbb{E}$.
\end{proof}

Recall that we can estimate the harmonic moments
of a positive random variable through its Laplace transform:

\begin{lem}[\cite{liu1999}, Lemma 4.4]
\label{MDL3.2} Let $X$ be a positive
random variable. For $0<a<\infty$, consider the following
statements:\newline
\begin{equation*}
\begin{array}{ll}
(i)\;\mathbb{E}X^{-a}<\infty; & (ii)\;\mathbb{E}e^{-tX}=O(t^{-a})(t\rightarrow\infty); \\
(iii)\;\mathbb{P}(X\leq x)=O(x^a)(x\rightarrow0); & (iv)\;\forall b\in(0,a),
\mathbb{E}X^{-b}<\infty. \\
\end{array}
\end{equation*}
Then the following implications hold: (i) $\Rightarrow$ (ii) $\Leftrightarrow$ (iii) $\Rightarrow$ (iv).
\end{lem}

Set
\begin{equation*}
\phi_\xi(t)=\mathbb{E}_\xi e^{-tW} \quad and\quad \phi(t)=\mathbb{E}\phi_\xi(t)=\mathbb{E}
e^{-tW}\;(t\geq 0).
\end{equation*}

\begin{lem}
\label{MDL3.3}  Assume (H). Then there exist constants $\beta\in(0,1)$ and $K\geq1$ such that
\begin{equation*}
\phi_\xi(t)\leq\beta \quad a.s. \qquad\forall t\geq\frac{1}{K}.
\end{equation*}
\end{lem}

\begin{proof}[Proof]
Let $p=1+\delta$. By a similar argument to the one used in the proof of (\cite{liu3}, Proposition 1.3), we have $\forall k\geq 0$,
\begin{eqnarray}\label{MDEL331}
	\mathbb{E}_\xi|W_{k+1}-W_k|^p\leq\left\{\begin{array}{ll}
	2^p\Pi_k^{1-p} \overline {m}_k(p) &\text{if $1<p\leq 2$,}\\
	(B_p)^p\Pi_k^{-p/2}\mathbb{E}_\xi W_k^{p/2} \overline {m}_k(p)  &\text{if $p>2$,}
	\end{array}\right.
\end{eqnarray}
where $B_p=2\sqrt{\lceil p/2\rceil}$ with $\lceil p/2\rceil = \min \{k\in \mathbb{N}: k\geq p/2\}$,  and
$\overline{m}_k(p) = \sum_{i=0}^\infty |\frac{i}{m_k} - 1 |^p p_i (\xi_k)  $.

The assumption (H) implies that $ \| \overline{m}_0(p) \|_\infty  =   \|\mathbb{E}_\xi|\frac{Z_1}{m_0}-1|^p\|_{\infty}<\infty$ and that
$\Pi_k\geq A_1^k$ a.s.. Using the inequality (\ref{MDEL331}) and an induction argument on $[p]$ (see \cite{liu3}, Proposition 1.3), we obtain
$$\mathbb{E}_\xi W^{1+\delta}=\sup_n\mathbb{E}_\xi W_n^p\leq C\qquad a.s.$$
for some constant $C$. In fact we shall only use the result for $\delta\leq 1$. Assume that $\delta\in(0,1]$, otherwise we consider $\min\{\delta,1\}$ instead of $\delta$. Notice that the function $\frac{e^{-x}-1+x}{x^{1+\delta}}$ is positive and bounded on $(0,\infty)$. So there exists a constant $C\geq 1$ such that
\begin{equation}\label{MDEL332}
e^{-x}\leq1-x+\frac{C }{1+\delta}x^{1+\delta} \qquad\forall x >0.
\end{equation}
Take $K:=\left(C\|\mathbb{E}_\xi W^{1+\delta}\|_{\infty}\right)^{1/\delta}\in[1,\infty)$.  By (\ref{MDEL332}), we obtain
\begin{eqnarray*}
\phi_\xi(t)=\mathbb{E}_\xi e^{-tW} &\leq &1-t+\frac{C}{1+\delta}t^{1+\delta} \mathbb{E}_\xi W^{1+\delta}\\
&\leq&1-t+\frac{K^\delta}{1+\delta}t^{1+\delta} \qquad a.s..
\end{eqnarray*}
 Let $g(t)=1-t+\frac{K^\delta}{1+\delta}t^{1+\delta}$. Obviously,
 $$\min_{t>0}g(t)=g(\frac{1}{K})=1-\frac{\delta}{K(1+\delta)}=:\beta\in(0,1)$$
(it can be seen that $\beta\geq \frac{1}{2}$).
Since $\phi_\xi(t)$ is decreasing, we have for $t\geq\frac{1}{K}$,
$$\phi_\xi(t)\leq\phi_\xi(\frac{1}{K})\leq g(\frac{1}{K})=\beta \qquad a.s..$$
\end{proof}

Denote
\begin{equation}
\underline{m}= \mbox{essinf} \; Z_1 = \inf \{j>0:\mathbb {P} (Z_1=j) >0\}.
\end{equation}
Notice that $\mathbb {P} (Z_1=j) = 0$ if and only if $\mathbb{P}(p_j(\xi_0)>0)=0$, so an alternative definition of $\underline{m}$  is
\begin{equation}
\underline{m} =\inf \{j>0:\mathbb{P}(p_{j}(\xi_0) >0)>0\}.
\end{equation}

The following Theorem gives an uniform bound for the quenched harmonic moments of $W$.

\begin{thm}
\label{MDT3.1} Assume (H).
\begin{itemize}
\item[](i) If $\|p_1\|_\infty<1$, then for some constants $a>0$ and $C>0$, we have a.s.,
\begin{equation*}
\phi _{\xi }(t)\leq Ct^{-a}\;(\forall t>0),\quad \mathbb{P}_{\xi }(W\leq
x)\leq C x^{a}\;(\forall x>0)\quad
\text{and}\quad
\mathbb{E}_{\xi }W^{-a}\leq C.
\end{equation*}

\item[](ii) If  $p_1=0$ a.s.,  then  a.s.
\begin{equation*}
\phi _{\xi }(t)\leq C_2\exp (-C_1t^{\gamma })\;(\forall
t>0),\quad \mathbb{P}_{\xi }(W\leq x)\leq C_2\exp
(-C_1x^{\frac{\gamma }{\gamma -1} })\;(\forall x>0),
\end{equation*}
and
$\mathbb{E}_{\xi }W^{-s}\leq C_{s}\;(\forall s>0)$,
where $\gamma =\frac{\log \underline{m}}{\log A}\in (0,1)$,
$C_1,C_2$ and $C_s$ are positive constants independent of $\xi$.
\end{itemize}
\end{thm}

\begin{proof}[Proof]
We only prove the results about $\phi_\xi(t)$, from which the
results about $\mathbb{P}_\xi(W\leq x)$ and $\mathbb{E}_\xi W^{-s}$ can be deduced by
Lemma \ref{MDL3.2} for (i), and by Tauberian theorems of exponential type (see \cite{liu5}) for (ii).

(i) It is clear that $\phi_\xi(t)$ satisfies the functional equation
\begin{equation}\label{MDE3.1}
\phi_\xi(t)=f_0(\phi_{T\xi}(\frac{t}{m_0}))
\end{equation}
(recall that $T^n\xi=(\xi_n, \xi_{n+1}, \cdots)$ if $\xi=(\xi_0, \xi_1,\cdots)$ and $n\geq 0$).
Hence a.s.,
\begin{eqnarray*}
 \phi_\xi(t)&\leq& \verb"" p_1(\xi_0)\phi_{T\xi}(\frac{t}{m_0})+(1-p_1(\xi_0))\phi_{T\xi}^2(\frac{t}{m_0})\nonumber\\
&\leq&\phi_{T\xi}(\frac{t}{m_0})\left(p_1(\xi_0)+(1-p_1(\xi_0))\phi_{T\xi}(\frac{t}{m_0})\right)\label{MDE3.2}\\
&\leq&\phi_{T\xi}(\frac{t}{m_0})\label{MDE3.3}.
\end{eqnarray*}
Similarly, we have a.s.,
\begin{eqnarray*}
\phi_{T\xi}(\frac{t}{m_0})&\leq&\phi_{T^2\xi}(\frac{t}{\Pi_2})
\left(p_1(\xi_1)+(1-p_1(\xi_1))\phi_{T^2\xi}(\frac{t}{\Pi_2})\right)\label{MDE3.4}
\leq\phi_{T^2\xi}(\frac{t}{\Pi_2}).\label{MDE3.5}
\end{eqnarray*}
Consequently, we get a.s.,
\begin{eqnarray*}
\phi_\xi(t)\leq\phi_{T^2\xi}(\frac{t}{\Pi_2})\left(p_1(\xi_1)+(1-p_1(\xi_1))\phi_{T^2\xi}(\frac{t}{\Pi_2})\right)
\left(p_1(\xi_0)+(1-p_1(\xi_0))\phi_{T^2\xi}(\frac{t}{\Pi_2})\right).
\end{eqnarray*}
By iteration, we obtain that $\forall n\geq1$, a.s.
\begin{equation}\label{MDE3.7}
\phi_\xi(t)\leq\phi_{T^n\xi}(\frac{t}{\Pi_n})\prod_{j=0}^{n-1}\left(p_1(\xi_j)+(1-p_1(\xi_j))\phi_{T^n\xi}(\frac{t}{\Pi_n})\right).
\end{equation}
By Lemma \ref{MDL3.3}, a.s.,
$\phi_{T^n\xi}(\frac{t}{\Pi_n})\leq\beta$ if $t\geq
\frac{A^n}{K}$ and $n\geq0$, since $\Pi_n\leq A^n$. Let $\bar{p}_1:=\|p_1\|_\infty$. As
$p_1(\xi_0)\leq\bar{p}_1$ a.s., it follows that a.s.,
$$\phi_\xi(t)\leq \beta\alpha^n \;\;\text{for $t\geq \frac{A^n}{K}$ and $n\geq0$,} $$
where $\alpha=\bar{p}_1+(1-\bar{p}_1)\beta\in(0,1)$. For $t\geq \frac{1}{K}$,
take $n_0=n_0(t)=[\frac{\log(Kt)}{\log A}]\geq 0$. Clearly,
$t\geq\frac{A^{n_0}}{K}$ and $\frac{\log(Kt)}{\log A}-1\leq
n_0\leq\frac{\log(Kt)}{\log A}$. Thus for $t\geq \frac{1}{K}$, a.s.
$$\phi_\xi(t)\leq \beta\alpha^{n_0}\leq\beta\alpha^{-1}(Kt)^{\frac{\log \alpha}{\log A}}=C_0t^{-a}\;,$$
where $C_0=\beta\alpha^{-1}K^{\frac{\log \alpha}{\log A}}>0$ and
$a=-\frac{\log \alpha}{\log A}>0$.  therefore we can choose a constant $C>0$ such that a.s.,
$\phi_\xi(t)\leq Ct^{-a} (\forall t>0)$. Thus the first part of the
theorem is proved.

(ii) By the equation (\ref{MDE3.1}),
$$\phi_\xi(t)=f_0(\phi_{T\xi}(\frac{t}{m_0}))\leq\left(\phi_{T\xi}(\frac{t}{m_0})\right)^{\underline{m}}\quad a.s..$$
By iteration, using Lemma \ref{MDL3.3}  we have
$$\phi_\xi(t)
\leq\left(\phi_{T^n\xi}(\frac{t}{\Pi_n})\right)^{\underline{m}^n}\leq\beta^{\underline{m}^n}\quad
a.s.\quad\text{for}\quad t\geq \frac{A^n}{K}.$$ Like the proof of
the first part, take $n_0=n_0(t)=[\frac{\log(Kt)}{\log A}]\geq 0$.
Then for $t\geq\frac{1}{K}$,
$$\phi_\xi(t)\leq\beta^{\underline{m}^{n_0}}\leq\exp{\left(\underline{m}^{-1}(\log \beta)(Kt)^{\frac{\log \underline{m} }{\log
A}}\right)}\leq\exp{(-C_1t^\gamma)}\quad a.s.,$$ where
$C_1=-\underline{m}^{-1}K^{\frac{\log \underline{m}}{\log
A}}\log\beta>0$ and $\gamma=\frac{\log \underline{m}}{\log
A}\in(0,1)$. It follows that we can choose $C_2>0$ such that a.s.,$\phi_\xi(t)\leq
C_2\exp(-C_1t^\gamma)$, $\forall t>0$. This completes the proof.
\end{proof}

We now study the annealed moments of $W$.
\begin{thm}
\label{MDT3.2} Assume (H).
\begin{itemize}
\item[](i) Then there exist constants $a>0$ and $C>0$ such that
\begin{equation}\label{MDET3.2.1}
\phi(t)\leq Ct^{-a}\;(\forall t>0),\;\; \mathbb{P}(W\leq x)\leq C
x^{a}\;(\forall x>0)\;\;\text{and}\;\;
\mathbb{E}W^{-s}< \infty\;(\forall s\in(0,a)).
\end{equation}

If additionally $\|p_1\|_\infty<1$, then for each $a>0$ with $\mathbb{E}p_1m_0^a<1$,  (\ref{MDET3.2.1}) holds for some constant $C>0$.

\item[](ii) If  $p_1=0$ a.s.,  then
\begin{equation*}
\phi(t)\leq C_2\exp(-C_1t^{\gamma})\;(\forall t>0),\quad \mathbb{P}(W\leq
x)\leq C_2\exp(-C_1x^{\frac{\gamma}{\gamma-1}})\;(\forall
x>0),
\end{equation*}
and $\mathbb{E}W^{-s}< \infty\;(\forall s>0)$,
where $\gamma=\frac{\log\underline{m}}{\log A}\in(0,1)$, and
$C_1, C_2$ are positive constants.
\end{itemize}
\end{thm}

Notice that when $\|p_1\|_\infty<1$, the conclusion that (\ref{MDET3.2.1}) holds for some $a>0$ is also a
direct consequence of Theorem \ref{MDT3.1}(i). But Theorem \ref{MDT3.2}(i) gives more precise information.
\\*

To prove Theorem \ref{MDT3.2}, we need the following lemma.

\begin{lem}[\cite{liu2}, Lemma 3.2]
\label{MDL3.4} Let
$\phi:\mathbb{R}_+\rightarrow \mathbb{R}_+$ be a bounded function
and let $A$ be a positive random variable such that for some
$0<p<1$, $t_0\geq0$ and all $t>t_0$,
\begin{equation*}
\phi(t)\leq p\mathbb{E}\phi(At).
\end{equation*}
If $pEA^{-a}<1$ for some $0<a<\infty$, then $\phi(t)=O(t^{-a})(t\rightarrow\infty)$.
\end{lem}

\begin{proof}[Proof of Theorem \ref{MDT3.2}]
Part (ii) is from Theorem \ref{MDT3.1}(ii) by taking the expectation $\mathbb{E}$.
For  part (i), we first consider the special case where $p_1\leq\bar{p}_1$ a.s. for some constant $\bar{p}_1<1$. By Theorem
\ref{MDT3.1}(i), we have $\phi_\xi(t)\leq
C_1t^{-a_1}\;a.s.\,(\forall t>0)$ for some positive constants $C_1$
and $a_1$. So for all $0<\epsilon<1$, there exists a constant
$t_\epsilon>0$ such that $\phi_\xi(t)\leq \epsilon$ a.s. for
$t\geq t_\epsilon$. Thus by (\ref{MDE3.2}),
\begin{equation}\label{MDE3.8}
\phi_\xi(t)\leq
(p_1+(1-p_1)\epsilon)\phi_{T\xi}(\frac{t}{m_0})\quad a.s.\quad
\text{if}\quad t\geq At_\epsilon.
\end{equation}
Notice that $\xi_0$ is independent of $T\xi$.
Taking the expectation in (\ref{MDE3.8}), we see that for $ t\geq
At_\epsilon$,
\begin{eqnarray*}
\phi(t)&\leq& \mathbb{E}\left[(p_1+(1-p_1)\epsilon)\phi_{T\xi}(\frac{t}{m_0})\right]\\
&=& \mathbb{E}\left[(p_1+(1-p_1)\epsilon)\mathbb{E}\left[\left.\phi_{T\xi}(\frac{t}{m_0})\right|\xi_0\right]\right]\\
&=& \mathbb{E}\left[(p_1+(1-p_1)\epsilon)\phi(\frac{t}{m_0})\right]=p_\epsilon \mathbb{E}\phi(\tilde{A}_\epsilon t),
\end{eqnarray*}
where $p_\epsilon=\mathbb{E}(p_1+(1-p_1)\epsilon)<1$ and
$\tilde{A}_\epsilon$ is a positive random variable whose
distribution is determined by
$$\mathbb{E}g(\tilde{A}_\epsilon)=\frac{1}{p_\epsilon}\mathbb{E}\left[(p_1+(1-p_1)\epsilon)g(\frac{1}{m_0})\right]$$
for all bounded and measurable function $g$. If $p_\epsilon
\mathbb{E}\tilde{A}_\epsilon^{-a}<1$, by Lemma \ref{MDL3.4}, we have
$\phi(t)=O(t^{-a}) (t\rightarrow\infty)$, or equivalently,
$\phi(t)\leq Ct^{-a}(\forall t>0)$ for some constant $C>0$. Since
$\mathbb{E}p_1m_0^a<1$, we can take $\epsilon>0$ small enough such that
$$p_\epsilon
\mathbb{E}\tilde{A}_\epsilon^{-a}=\mathbb{E}\left[(p_1+(1-p_1)\epsilon)m_0^a\right]<1.$$
Therefore we have proved that $\phi(t)=O(t^{-a})$ whenever $\|p_1\|_\infty<1$ and $\mathbb{E}p_1m_0^a<1 (a>0)$.
Now consider the general case where $\|p_1\|_\infty$ may be $1$. By Lemma \ref{MDL3.3}, we have
$\phi_\xi(t)\leq \beta$ a.s. for $t\geq t_\beta=\frac{1}{K}$. So we can
repeat the proof above with $\beta$ in place of $\epsilon$, showing that if $a>0$ small enough such that
$$\mathbb{E}[(p_1+(1-p_1)\beta)m_0^a]\leq A^a(\mathbb{E}p_1+(1-\mathbb{E}p_1)\beta)<1,$$
then $\phi(t)=O(t^{-a})$.   Now we have proved the results about $\phi(t)$. By
Lemma \ref{MDL3.2}, we  obtain  the
results about $\mathbb{P}(W\leq x)$ and $\mathbb{E}W^{-s}$.
\end{proof}

We now prove our main result on the harmonic moments of $W$ already
stated in the introduction at the beginning of this paper .

\begin{proof}[Proof of Theorem \ref{MDT1.3}]
If $\mathbb{E}p_1m_0^a<1$, then there exists $\epsilon>0$ such that
$\mathbb{E}p_1m_0^{a+\epsilon}<1$. So by Theorem \ref{MDT3.2}(i),
$\mathbb{E}W^{-a}<\infty$. Conversely, assume that $a>0$ and $\mathbb{E}W^{-a}<\infty$. Notice that
$$W=\frac{1}{m_0}\sum_{i=1}^{Z_1}W_i^{(1)}\qquad a.s.,$$
where $\left(W_i^{(1)}\right)_{i\geq1}$, when $\xi$ is given, are
conditionally independent copies of $W^{(1)}$ whose distribution is
$\mathbb{P}_\xi(W^{(1)}\in \cdot)=\mathbb{P}_{T\xi}(W\in\cdot)$. Since
$\mathbb{P}(Z_1\geq2)>0$, we have
$$\mathbb{E}W^{-a}>\mathbb{E}m_0^a\left(W_1^{(1)}\right)^{-a}\mathbf{1}_{\{Z_1=1\}}=\mathbb{E}p_1m_0^a\mathbb{E}W^{-a}.$$
Therefore $\mathbb{E}p_1m_0^a<1$.
\end{proof}

\section{Moments of $Z_n$ and large deviations for $\log Z_n$}\label{MDS4}
We first recall some preliminary  results for the existence of moments of $W$.

Guivarc'h \& Liu \cite{liu1} gave a sufficient and necessary condition for the existence of moments of positive orders of $W$: for $s>1$,
\begin{equation}\label{MDEMa}
0<\mathbb{E}W^s<\infty\quad\text{if and only if }\quad \mathbb{E}\left(\frac{Z_1}{m_0}
\right)^s<\infty \;\text{and}\; \mathbb{E}m_0^{1-s}<1.
\end{equation}
In particular, if $p_0=0$ a.s. and $\mathbb{E}Z_1^s<\infty$ for all $s>1$, then
$0<\mathbb{E}W^s<\infty$ for all $s>0$.

For the existence of moments of negative orders of $W$, Theorem \ref {MDT1.3} shows that,  assuming (H) and $\|p_1\|_\infty<1$,
we have for $s>0$,
\begin{equation}\label{MDEMb}
\mathbb{E}W^{-s}<\infty\quad\text{if and only if }\quad \mathbb{E}p_1m_0^s<1.
\end{equation}
In particular, if $p_0=p_1=0$ a.s., it is clear that $\mathbb{E}W^{-s}<\infty$, for all $s>0$.

These results will be applied in the proof of Theorem \ref{MDT5.3}.

\begin{proof}[Proof of Theorem \ref{MDT5.3}]
Denote the distribution of $\xi_0$ by $\tau_0$. Fix $t\in\mathbb{R}$ and define a new distribution $\tilde{\tau}_0$ as
$$\tilde{\tau}_0(dx)=\frac{m(x)^t\tau_0(dx)}{\mathbb{E}m_0^t},$$
where $m(x)=\mathbb{E}[Z_1|\xi_0=x]=\sum_{i=0}^{\infty}ip_i(x)$.  Consider the new branching process in a random environment whose environment
distribution
is $\tilde{\tau}=\tilde{\tau}_0^{\otimes \mathbb{N}}$ instead of $\tau=\tau_0^{\otimes \mathbb{N}}$. The corresponding probability and
expectation are denoted by
$\tilde{\mathbb{P}}=\mathbb{P}_\xi \otimes \tilde \tau$ and $\mathbb{\tilde{E}}$, respectively. Then
$$\frac{\mathbb{E}Z_n^t}{\left(\mathbb{E}m_0^t\right)^n}=\tilde{\mathbb{E}}W_n^t.$$
It is easy to see that under $\tilde{\mathbb{P}}$, we still have $p_0=0$ a.s.. Moreover, if (H) holds and $\|p_1\|_\infty<1$, then the same hold under $\tilde{\mathbb{P}}$.   Notice that
$$\tilde{\mathbb{E}}\log m_0=\frac{\mathbb{E}m_0^t\log m_0}{\mathbb{E}m_0^t}\in(0,\infty].$$
We distinguish three cases as considered  in the theorem.

 (i) If $t\in(0,1]$ and $\mathbb{E}m_0^{t-1}Z_1\log^+Z_1<\infty$, then
$$\tilde{\mathbb{E}}\frac{Z_1}{m_0}\log^+Z_1=\frac{\mathbb{E}m_0^{t-1}Z_1\log^+Z_1}{\mathbb{E}m_0^t}<\infty,$$
so that $W_n\rightarrow W$ in $L^1$ under $\tilde{\mathbb{P}}$ (cf.  Athreya \& Karlin (1971) or Tanny (1988)). Therefore,
\begin{equation}\label{MDEP1.4}
\lim_{n\rightarrow\infty}\tilde{\mathbb{E}}W_n^t=\tilde{\mathbb{E}}W^t\in(0,\infty).
\end{equation}

(ii) If $t>1$ and $\mathbb{E}Z_1^t<\infty$, then
$$\tilde{\mathbb{E}}\left(\frac{Z_1}{m_0}\right)^t=\frac{\mathbb{E}Z_1^t}{\mathbb{E}m_0^t}<\infty  \quad a.s. \quad\text{under $\tilde{\mathbb{P}}$,}$$
so that  $W_n\rightarrow W$ in $L^t$ under $\tilde{\mathbb{P}}$ (cf. (\ref{MDEMa})).

(iii) If $t<0$, $\mathbb{E}p_1<\mathbb{E}m_0^t$, $\|p_1\|_\infty<1$ and (H) holds, then
$$\tilde{\mathbb{E}}p_1m_0^{-t}=\frac{\mathbb{E}p_1}{\mathbb{E}m_0^t}<1,$$
so that  $\tilde{\mathbb{E}}W^t<\infty$ from Theorem \ref{MDT1.3}. Using Lemma \ref{MDL3.1}, we obtain again (\ref{MDEP1.4}).

Therefore we have proved  Theorem \ref{MDT5.3} with $C(t)=\tilde{\mathbb{E}}W^t$.
 \end{proof}

Using Theorem \ref{MDT5.3}, we can  easily prove Theorem \ref{MDT5.2}.

\begin{proof}[Proof of Theorem \ref{MDT5.2}]

It is clear that the hypothesis of Theorem \ref{MDT5.2} ensures that $\mathbb{E}Z_1^t<\infty$ for all $t\in \mathbb{R}$. Hence by
Theorem \ref{MDT5.3},
$$\lim_{n\rightarrow\infty}\frac{\mathbb{E}Z_n^t}{\left(\mathbb{E}m_0^t\right)^n}=C(t)\in(0,\infty)\qquad \forall t \in \mathbb{R},$$
which implies that
\begin{equation}\label{MDE5.7}
\lim_{n\rightarrow\infty}\frac{1}{n}\log \mathbb{E}Z_n^t=\log \mathbb{E}m_0^t=\Lambda (t)\qquad \forall t \in \mathbb{R}.
\end{equation}
 Notice that the Laplace transform of $\log Z_n$ is $\mathbb{E}e^{t\log Z_n}=\mathbb{E}Z_n^t$. As $\Lambda(t)$ is finite and derivable everywhere,
  from (\ref{MDE5.7}) and the G\"artner-Ellis theorem (\cite{z}, p.52, Exercise 2.3.20),
we immediately obtain Theorem \ref{MDT5.2}.
\end{proof}

Theorem \ref{MDT5.3} can also be used to study the large deviation probabilities $\mathbb{P}\left(\frac{\log Z_n}{n}\geq x\right)$ (resp. $\mathbb{P}\left(\frac{\log Z_n}{n}\leq x\right)$)
for a finite interval of $x$, when $\mathbb{E}W^{a}$ (resp. $\mathbb{E}W^{-a}$) ($a>0$) exists only in a finite interval of $a$. To this end we shall use the following version of the G\"artner-Ellis theorem adapted to the study of tail probabilities.

\begin{lem}[\cite{liu4}, Theorem 6.1]\label{MDL5.3}
Let $(\mu_n)$ be a family of probability distribution on $\mathbb{R}$ and
let $(a_n)$ be a sequence of positive numbers satisfying $a_n\rightarrow\infty$.  Assume that for some $t_0\in[0,\infty]$ and for every $t\in[0,t_0)$, as $n\rightarrow\infty$,
$$l_n(t):=\frac{1}{a_n}\log \int e^{a_ntx}\mu_n(dx)\rightarrow l(t)<\infty.$$
For $x\in\mathbb{R}$, set
$$l^*(x)=\sup\{tx-l(t); t\in[0,t_0)\}.$$
If $l$ is continuously differentiable on $(0,t_0)$, then for all $x\in(l'(0+), l'(t_0-))$ (where $l'(x\pm)=\lim_{y\rightarrow x\pm}l'(y)$),
$$\lim_{n\rightarrow\infty}\frac{1}{a_n}\log \mu_n([x, \infty))=-l^*(x).$$
\end{lem}

From  Theorem \ref{MDT5.3} and Lemma \ref{MDL5.3}, we immediately obtain the following theorem.

\begin{thm}\label{MDT5.4}
Let $a\in\mathbb{R}$.
\begin{itemize}
\item[](i) Let $a>0$. If $a\in(0,1]$ and $\mathbb{E}m_0^{a-1}Z_1\log^+Z_1<\infty$, or $a>1$ and $\mathbb{E}Z_1^a<\infty$, then
\begin{equation}\label{MDEL4.1a}
\lim_{n\rightarrow\infty}\frac{1}{n}\log \mathbb{P}\left(\frac{\log Z_n}{n}\geq x\right)=-\Lambda^*(x),
\quad \forall x\in(\mathbb{E}\log m_0, \Lambda'(a)).
\end{equation}

\item[](ii) Let  $a<0$. Assume  (H) and $\|p_1\|_\infty<1$.  If $\mathbb{E}p_1<\mathbb{E}m_0^a$, then
\begin{equation}\label{MDEL4.1b}
\lim_{n\rightarrow\infty}\frac{1}{n}\log \mathbb{P}\left(\frac{\log Z_n}{n}\leq x\right)=-\Lambda^*(x), \quad \forall x\in(\Lambda'(a), \mathbb{E}\log m_0).
\end{equation}
\end{itemize}
\end{thm}

  If $\mathbb{E}Z_1^a<\infty$ for all $a>1$ (resp.  $p_1=0\;a.s.$), then Theorem \ref{MDT5.4} suggests that the limit in (\ref{MDEL4.1a}) (resp. (\ref{MDEL4.1b})) would hold for any
$x>\mathbb{E}\log m_0$ (resp. $x<\mathbb{E}\log m_0$). This leads to the following theorem  which is more precise than Corollary \ref{MDC1.1}. It was proved by Bansaye \& Berestycki \cite{ba} when (H) holds with $\delta=1$.

\begin{thm}
\label{MDT5.1} (i) If $\mathbb{E}Z_1^s<\infty$ for all $s>1$, then
\begin{equation*}
\lim_{n\rightarrow\infty}\frac{1}{n}\log \mathbb{P}\left(\frac{\log
Z_n}{n}\geq x\right)=-\Lambda^*(x)\quad \text{for}\;x> \mathbb{E}\log m_0.
\end{equation*}

(ii) Assume (H) and $p_1=0$ a.s., then

\begin{equation*}
\lim_{n\rightarrow\infty}\frac{1}{n}\log \mathbb{P}\left(\frac{\log
Z_n}{n}\leq x\right)=-\Lambda^*(x)\quad \text{for}\;x< \mathbb{E}\log
m_0,
\end{equation*}
\end{thm}

If $\Lambda'(\infty)=\infty$ and $\Lambda'(-\infty)=0$, then Theorem \ref{MDT5.1} can be directly deduced from Theorem \ref{MDT5.4}.  But it is possible that
$\Lambda'(\infty)<\infty $ or $\Lambda'(-\infty)>0$. So we will give a direct proof of Theorem \ref{MDT5.1}, following \cite{ba}.

According to the large deviation principle for i.i.d.  random
variables, we have
\begin{equation}
\lim_{n\rightarrow \infty }\frac{1}{n}\log \mathbb{P}\left( \frac{\log
\Pi_{n}}{n}\leq x\right) =-\Lambda ^{\ast }(x)\quad
\text{for}\;x\leq \mathbb{E}\log m_{0}, \label{MDE5.1}
\end{equation}
\begin{equation}
\lim_{n\rightarrow \infty }\frac{1}{n}\log \mathbb{P}\left( \frac{\log
\Pi_{n}}{n}\geq x\right) =-\Lambda ^{\ast }(x)\quad \text{for}\;x\geq
\mathbb{E}\log m_{0}. \label{MDE5.2}
\end{equation}

Lemma \ref{MDL5.1} below gives the lower bound for both the lower
and upper deviations.

\begin{lem}[\cite{ba}, Proposition 1]
\label{MDL5.1} Assume (\ref{MDE1.0}). Then
\begin{equation}  \label{MDE5.3}
\liminf_{n\rightarrow\infty}\frac{1}{n}\log \mathbb{P}\left(\frac{\log
Z_n}{n}\leq x\right)\geq-\Lambda^*(x)\quad \text{for}\;x\leq \mathbb{E}\log
m_0,
\end{equation}
\begin{equation}  \label{MDE5.4}
\liminf_{n\rightarrow\infty}\frac{1}{n}\log \mathbb{P}\left(\frac{\log
Z_n}{n}\geq x\right)\geq-\Lambda^*(x)\quad \text{for}\;x\geq \mathbb{E}\log
m_0.
\end{equation}
\end{lem}

We remark that in Lemma \ref{MDL5.1}, the original moment condition in (\cite{ba}, Proposition 1), namely, $\mathbb{E}\left(\frac{Z_1}{m_0}\right)^s<\infty$ for some $s>1$,  is weaken to $\mathbb{E}\frac{Z_1}{m_0}\log^+Z_1<\infty$.
\\*

The following lemma gives the upper bound for both the lower
and upper deviations.

\begin{lem}
\label{MDL5.2}  (i) If $\mathbb{E}W^{- s}<\infty $ for all $s>1$, then
\begin{equation}
\limsup_{n\rightarrow \infty }\frac{1}{n}\log \mathbb{P}\left( \frac{\log Z_{n}}{n}
\leq x\right) \leq -\Lambda ^{\ast }(x)\quad \text{for}\;x<
\mathbb{E}\log m_{0}. \label{MDE5.8}
\end{equation}

(ii) If $\mathbb{E}W^{s}<\infty $ for all $s>0$, then
\begin{equation}
\limsup_{n\rightarrow \infty }\frac{1}{n}\log \mathbb{P}\left( \frac{\log Z_{n}}{n}
\geq x\right) \leq -\Lambda ^{\ast }(x)\quad \text{for}\;x> \mathbb{E}\log
m_{0}. \label{MDE5.9}
\end{equation}
\end{lem}

The inequality (\ref{MDE5.9}) was proved by Bansaye \& Berestycki \cite{ba}. For readers' convenience, we shall prove simultaneously
(\ref{MDE5.9}) and (\ref{MDE5.8}).

\begin{proof}[Proof of Lemma \ref{MDL5.2}]
 By the decomposition (\ref{MDE0}),  for $x\in\mathbb{R}$,
$\epsilon>0$ and $s>0$, we have
\begin{eqnarray*}
\mathbb{P}\left(\frac{\log Z_n}{n}\leq x\right)  &\leq&\mathbb{P}\left(\frac{\log
\Pi_n}{n}\leq
x+\epsilon\right)+\mathbb{P}\left(\frac{\log W_n}{n}\leq -\epsilon\right).
 \end{eqnarray*}
 By Markov's inequality and Lemma \ref{MDL3.1},
 $$\mathbb{P}\left(\frac{\log W_n}{n}\leq -\epsilon\right)\leq\frac{\mathbb{E}W_n^{-s}}{e^{s\epsilon n}}
 \leq\frac{\mathbb{E}W^{-s}}{e^{s\epsilon n}.}$$
Thus
\begin{eqnarray*}
\limsup_{n\rightarrow\infty}\frac{1}{n}\log \mathbb{P}\left(\frac{\log
Z_n}{n}\leq
x\right)&\leq&\max\{\limsup_{n\rightarrow\infty}\frac{1}{n}\log
\mathbb{P}\left(\frac{\log \Pi_n}{n}\leq x+\epsilon\right), -s\epsilon\}\\
&=&\max\{-\Lambda^*(x+\epsilon), -s\epsilon\}.
\end{eqnarray*}
Letting $s\rightarrow\infty$ and  $\epsilon\rightarrow0$, we
obtain (\ref{MDE5.8}). For (\ref{MDE5.9}),  we use a similar argument. For
$\epsilon>0$ and $s>1$,
\begin{eqnarray*}
\mathbb{P}\left(\frac{\log Z_n}{n}\geq x\right)  &\leq&\mathbb{P}\left(\frac{\log
\Pi_n}{n}\geq
x-\epsilon\right)+\mathbb{P}\left(\frac{\log W_n}{n}\geq \epsilon\right)\\
&\leq&\mathbb{P}\left(\frac{\log \Pi_n}{n}\geq
x-\epsilon\right)+\frac{\mathbb{E}W^{s}}{e^{s\epsilon n}}.
 \end{eqnarray*}
Thus
\begin{eqnarray*}
\limsup_{n\rightarrow\infty}\frac{1}{n}\log \mathbb{P}\left(\frac{\log
Z_n}{n}\geq
x\right)&\leq&\max\{\limsup_{n\rightarrow\infty}\frac{1}{n}\log
\mathbb{P}\left(\frac{\log \Pi_n}{n}\geq x-\epsilon\right), -s\epsilon\}\\
&=&\max\{-\Lambda^*(x-\epsilon), -s\epsilon\}.
\end{eqnarray*}
Again letting $s\rightarrow\infty$ and $\epsilon\rightarrow0$,  we
obtain (\ref{MDE5.9}).
\end{proof}

\begin{proof}[Proof of Theorem \ref{MDT5.1}]
It is just a combination of Lemmas \ref{MDL5.1} and
\ref{MDL5.2}.
\end{proof}

Notice that Theorem \ref{MDT5.1} implies Corollary \ref{MDC1.1}.  By Lemma \ref{MDL4.2},  we see that Corollary \ref{MDC1.1} is in fact equivalent to Theorem \ref{MDT5.2}.
So the direct proof of Theorem \ref{MDT5.1} leads to an alternative proof of Theorem \ref{MDT5.2}.

\section{Moderate deviations for $\log Z_n$}\label{MDS5}
Now we turn to the proof of moderate deviation principle (Theorem \ref{MDT4.3}).
Similar to the proof of large deviation principle (Theorem \ref{MDT5.2}), we can
study the convergence rate of $\frac{\log Z_{n}}{n}$ by considering
those of $\frac{\log \Pi_{n}}{n}$. Recall that $(a_n)$ is a sequence of positive numbers satisfying (\ref{an}).
 Let
$$S_n:=\log \Pi_n-n\mathbb{E}\log m_0\quad\text{ and }\quad\bar{\Lambda}_n(t)=\log\mathbb{ E }\exp\left({\frac{tS_n}{a_n}}\right).$$
By the classic moderate deviation results for i.i.d.  random variables (see \cite{z}, Theorem 3.7.1 and its proof), it is known that,
if $f(t )=\mathbb{E}m_{0}^{t }<\infty $ in
a neighborhood of the origin,  then
\begin{equation}\label{MDE41}
\lim_{n\rightarrow\infty}\frac{n}{a_n^2}\bar{\Lambda}_n(\frac{a_n^2}{n}t)= \frac{1}{2}\sigma ^{2}t^2,
\end{equation}
and for any measurable subset $B$ of $ \mathbb{R}$,
  \begin{eqnarray}\label{MDEM}
  -\inf_{x\in B^o}\frac{x^2}{2\sigma^2}&\leq&\liminf_{n\rightarrow\infty}\frac{n}{a_n^2}\log \mathbb{P}\left(\frac{\log \Pi_n-n\mathbb{E}\log m_0}{a_n}\in B\right)\nonumber\\
  &\leq&\limsup_{n\rightarrow\infty} \frac{n}{a_n^2}\log \mathbb{P}\left(\frac{\log \Pi_n-n\mathbb{E}\log m_0}{a_n}\in B\right)
  \leq  -\inf_{x\in \bar{B}}\frac{x^2}{2\sigma^2} . \end{eqnarray}

\begin{lem}\label{MDL4.3} Let $t\in \mathbb{R}$.
\begin{itemize}
\item[](i) If (H) holds and $\|p_1\|_\infty<1$,   then for all $t<0$,
\begin{equation}\label{MDE4.14}
\lim_{n\rightarrow\infty}\frac{\mathbb{E}Z_n^{\frac{a_n}{n}t}}{\mathbb{E}\Pi_n^{\frac{a_n}{n}t}}=1.
\end{equation}

\item[](ii) If (H) holds, then there is a constant $c>0$ such that for all $t>0$,
\begin{equation}\label{MDE4.15}
c\leq\liminf_{n\rightarrow\infty}\frac{\mathbb{E}Z_n^{\frac{a_n}{n}t}}{\mathbb{E}\Pi_n^{\frac{a_n}{n}t}}
\leq\limsup_{n\rightarrow\infty}\frac{\mathbb{E}Z_n^{\frac{a_n}{n}t}}{\mathbb{E}\Pi_n^{\frac{a_n}{n}t}}\leq1.
\end{equation}
\end{itemize}
\end{lem}

\begin{proof}[Proof]
(i)  Let $t_n=\frac{a_n}{n}t$.  For $t<0$, we have $t_n<0$. By Jensen's inequality,
$$\mathbb{E}_\xi W_n^{t_n}\geq (\mathbb{E}_\xi W_n)^{t_n}=1\qquad a.s..$$
Thus
\begin{equation}\label{MDE4.16}
\mathbb{E}Z_n^{t_n}=\mathbb{E}\Pi_n^{t_n}\mathbb{E}_\xi W_n^{t_n}\geq \mathbb{E }\Pi_n^{t_n},
\end{equation}
which leads to
$$\liminf_{n\rightarrow\infty}\frac{\mathbb{E}Z_n^{t_n}}{\mathbb{E}\Pi_n^{t_n}}\geq 1.$$
On the other hand,
if (H) holds and $\|p_1\|_\infty<1$,
then by Theorem \ref{MDT3.1}, we have $\mathbb{E}_\xi W^{-s}\leq C_s$ a.s. for some constants $s>0$ and $C_s>0$.
Noticing that $-t_n/s\in(0,1)$ for $n$ large enough and  that by Lemma \ref{MDL3.1}, $\mathbb{E}_\xi W_n^{-s}\leq \mathbb{E}_\xi W^{-s}$ a.s., again by Jensen's inequality, we have
$$\mathbb{E}_\xi W_n^{t_n}=\mathbb{E}_\xi(W_n^{-s})^{-t_n/s}\leq (\mathbb{E}_\xi W_n^{-s})^{-t_n/s}\leq (\mathbb{E}_\xi W^{-s})^{-t_n/s}\leq C_s^{-t_n/s},$$
so that
$$\mathbb{E}Z_n^{t_n}\leq C_s^{-t_n/s}\mathbb{E}\Pi_n^{t_n}.$$
Letting $n\rightarrow\infty$, we obtain
$$\limsup_{n\rightarrow\infty}\frac{\mathbb{E}Z_n^{t_n}}{\mathbb{E}\Pi_n^{t_n}}\leq 1.$$

(ii) For $t>0$,  we have $t_n=\frac{a_n}{n}t\in(0,1)$ for $n$ large enough, so by Jensen's inequality,
$$\mathbb{E}_\xi W_n^{t_n}\leq (\mathbb{E}_\xi W_n)^{t_n}=1\qquad a.s..$$ Thus
$$\limsup_{n\rightarrow\infty}\frac{\mathbb{E}Z_n^{t_n}}{\mathbb{E}\Pi_n^{t_n}}\leq 1.$$
On the other hand,  from the proof Lemma \ref{MDL3.3}, we know that  the assumption (H) ensures that
 $\mathbb{E}_\xi W^{s}\leq C_s$ a.s.   for $1<s\leq 1+\delta$ and some constant $C_s>0$.
By H\"older's inequality,
\begin{eqnarray}
1=\mathbb{E}_\xi W_n&\leq& \mathbb{E}_\xi W_n^{t_n/p}W_n^{1-t_n/p}\nonumber\\
&\leq&\left(\mathbb{E}_\xi W_n^{t_n}\right)^{1/p}\left(\mathbb{E}_\xi W_n^{(1-t_n/p)q}\right)^{1/q}\qquad a.s.,\label{MDE4.17}
\end{eqnarray}
for $p,q>1$, $1/p+1/q=1$. Take $p=p(n)=\frac{s-t_n}{s-1}$ and $q=q(n)=\frac{s-t_n}{1-t_n}$,  so that $(1-t_n/p)q=s$
 and $p/q=\frac{1-t_n}{s-1}$. Notice that by Lemma \ref{MDL3.1}, $\mathbb{E}_\xi W_n^{-s}\leq \mathbb{E}_\xi W^{-s}$ a.s.. We deduce from (\ref{MDE4.17}) that
$$\mathbb{E}_\xi W_n^{t_n}\geq \left(\mathbb{E}_\xi W_n^{s}\right)^{-\frac{1-t_n}{s-1}}\geq
\left(\mathbb{E}_\xi W^{s}\right)^{-\frac{1-t_n}{s-1}}\geq   C_s^{-\frac{1-t_n}{s-1}}.$$
Thus
$$\mathbb{E}Z_n^{t_n}\geq C_s^{-\frac{1-t_n}{s-1}}\mathbb{E}\Pi_n^{t_n}.$$
Letting $n\rightarrow\infty$, we obtain
$$\liminf_{n\rightarrow\infty}\frac{\mathbb{E}Z_n^{t_n}}{\mathbb{E}\Pi_n^{t_n}}\geq c,$$
where $c= C_s^{-\frac{1}{s-1}}\in(0,1]$. This completes the proof.
\end{proof}

\begin{thm}
Let $\Lambda_n(t)=\log \mathbb{E} \exp{\left(\frac{\log Z_n-n \mathbb{E}\log m_0}{a_n}t\right)}$
and $\bar{\Lambda}_n(t)=\log \mathbb{E }\exp\left({\frac{tS_n}{a_n}}\right)$.  If (H) holds, then
\begin{equation}\label{MDE4.18}
\lim_{n\rightarrow\infty}\frac{\Lambda_n(\frac{a_n^2}{n}t)}{\bar{\Lambda}_n(\frac{a_n^2}{n}t)}=1,\qquad \forall t\neq 0
\end{equation}
and
\begin{equation}\label{MDE4.19}
\lim_{n\rightarrow\infty}\frac{\log \mathbb{E}Z_n^{\frac{a_n}{n}t}}{\log \mathbb{E}\Pi_n^{\frac{a_n}{n}t}}=1,\qquad \forall t\neq 0.
\end{equation}
\end{thm}

\begin{proof}[Proof]
We only need prove (\ref{MDE4.18}), which implies (\ref{MDE4.19}). For $t>0$,  (\ref{MDE4.18}) is a direct consequence of Lemma \ref{MDL4.3}(ii).
 For $t<0$,  if additionally $\|p_1\|_\infty<1$, then (\ref{MDE4.18}) is also a direct consequence of Lemma \ref{MDL4.3}(i); we shall prove that the condition
 $\|p_1\|_\infty<1$ is not needed for (\ref{MDE4.18}) to hold. Assume (H) and let $t<0$. Notice that
 (\ref{MDE4.16}) implies that
 $$\liminf_{n\rightarrow\infty}\frac{\Lambda_n(\frac{a_n^2}{n}t)}{\bar{\Lambda}_n(\frac{a_n^2}{n}t)}\geq1.$$
 It remains to show that
 \begin{equation}\label{MDE4.20}
\limsup_{n\rightarrow\infty}\frac{\Lambda_n(\frac{a_n^2}{n}t)}{\bar{\Lambda}_n(\frac{a_n^2}{n}t)}\leq 1.
\end{equation}
By H\"older's inequality,
\begin{eqnarray*}
\exp{\left(\Lambda_n(\frac{a_n^2}{n}t)\right)}&=&\mathbb{E}\exp{\left(\frac{a_n}{n}t(\log Z_n -n \mathbb{E}\log m_0)\right)}\\
&=&\mathbb{E}e^{\frac{a_n}{n}tS_n}W_n^{\frac{a_n}{n}t}\\
&\leq&\left(\mathbb{E}e^{\frac{a_n}{n}ptS_n}\right)^{1/p}\left(\mathbb{E} W_n^{\frac{a_n}{n}tq}\right)^{1/q}\\
&\leq&\exp{\left(\frac{1}{p}\bar{\Lambda}_n(\frac{a_n^2}{n}pt)\right)}\left(\mathbb{E} W_n^{\frac{a_n}{n}tq}\right)^{1/q},
\end{eqnarray*}
where $p,q>1$ are constants satisfying $1/p+1/q=1$. By Theorem \ref{MDT3.2}, there exists $s>0$ such that $\mathbb{E}W^{-s}<\infty$.
Noticing that $t_n q>-s$ for n large, we have
$$\mathbb{E}W_n^{t_nq}\leq 1+\mathbb{E}W_n^{-s}\leq 1+\mathbb{E} W^{-s}.$$
Hence for $n$ large enough,
$$\Lambda_n(\frac{a_n^2}{n}t)\leq \frac{1}{p}\bar{\Lambda}_n(\frac{a_n^2}{n}pt)+\frac{1}{q}\log (1+\mathbb{E}W^{-s}).$$
Therefore, considering (\ref{MDE41}), we have
$$\limsup_{n\rightarrow\infty}\frac{\Lambda_n(\frac{a_n^2}{n}t)}{\bar{\Lambda}_n(\frac{a_n^2}{n}t)}
\leq  \frac{1}{p}\frac{\frac{1}{2}\sigma^2p^2t^2}{\frac{1}{2}\sigma^2t^2}= p.$$
Letting $p\rightarrow 1$, (\ref{MDE4.20}) is proved.
\end{proof}

\begin{proof}[Proof of Theorem \ref{MDT4.3}]
From (\ref{MDE4.18}) and (\ref{MDE41}) we have
$$\lim_{n\rightarrow\infty}\frac{n}{a_n^2}\Lambda_n(\frac{a_n^2}{n}t)
=\lim_{n\rightarrow\infty}\frac{n}{a_n^2}\bar{\Lambda}_n(\frac{a_n^2}{n}t)= \frac{1}{2}\sigma ^{2}t^2.$$
Applying the G\"artner-Ellis theorem (\cite{z}, p.52, Exercise 2.3.20), we obtain Theorem \ref{MDT4.3}.
\end{proof}

The following theorem about the tail probabilities is a direct consequence of Theorem \ref{MDT4.3}.

\begin{thm}
\label{MDT1.2} Assume $(H)$ and write  $\sigma ^{2}=\emph{var}(\log
m_{0})\in(0,\infty)$. Then for all $x>0$,
\begin{equation}
\lim_{n\rightarrow \infty }\frac{n}{a_{n}^{2}}\log \mathbb{P}\left(
\frac{\log Z_{n}-n\mathbb{E}\log m_{0}}{a_{n}}\leq -x\right)
=-\frac{x^{2}}{2\sigma ^{2}}, \label{MDE1.3}
\end{equation}
\begin{equation}
\lim_{n\rightarrow \infty }\frac{n}{a_{n}^{2}}\log \mathbb{P}\left(
\frac{\log Z_{n}-n\mathbb{E}\log m_{0}}{a_{n}}\geq x\right)
=-\frac{x^{2}}{2\sigma ^{2}}. \label{MDE1.4}
\end{equation}
\end{thm}

It is also possible to give a direct proof of Theorem \ref{MDT1.2}. We shall give such a proof in the following,
as it will give additional one-side results on the tail probabilities under weaker assumptions.

\begin{lem}
\label{MDT4.1}  If $f(t)=\mathbb{E}m_0^t<\infty$ in a
neighborhood of the origin, then for all $x>0$ ,

\begin{equation}  \label{MDE4.3}
\liminf_{n\rightarrow\infty}\frac{n}{a_n^2}\log \mathbb{P}\left(\frac{\log
Z_n-n\mathbb{E}\log m_0}{a_n}\leq -x\right)\geq-\frac{x^2}{2\sigma^2},
\end{equation}
 \begin{equation}  \label{MDE4.4}
\limsup_{n\rightarrow\infty}\frac{n}{a_n^2}\log \mathbb{P}\left(\frac{\log
Z_n-n\mathbb{E}\log m_0}{a_n}\geq x\right)\leq-\frac{x^2}{2\sigma^2}.
\end{equation}
\end{lem}

\begin{proof}[Proof]
Let $x>0$.
By (\ref{MDEM}), the moderate deviation principle for $\log \Pi_n$, we have
\begin{equation}
\lim_{n\rightarrow \infty }\frac{n}{a_{n}^{2}}\log \mathbb{P}\left(
\frac{\log \Pi_{n}-n\mathbb{E}\log m_{0}}{a_{n}}\leq -x\right)
=-\frac{x^{2}}{2\sigma ^{2}} \label{MDE4.1}
\end{equation}
and
\begin{equation}
\lim_{n\rightarrow \infty }\frac{n}{a_{n}^{2}}\log \mathbb{P}\left(
\frac{\log \Pi_{n}-n\mathbb{E}\log m_{0}}{a_{n}}\geq x\right)
=-\frac{x^{2}}{2\sigma ^{2}}.\label{MDE4.2}
\end{equation}
 For every $\epsilon>0$,
\begin{eqnarray*}
&&\mathbb{P}\left(\frac{\log Z_n-n\mathbb{E}\log m_0}{a_n}\leq -x \right)\\
&\geq&\mathbb{P}\left(\frac{\log \Pi_n-n\mathbb{E}\log m_0}{a_n}\leq -x-\epsilon\right)-\mathbb{P}(W_n\geq e^{a_n\epsilon})\\
&=&:u_n-v_n=u_n(1-v_n/u_n ).
\end{eqnarray*}
By (\ref{MDE4.1}), we have $\forall\delta'>0$, for $n$ large enough,
$$u_n\geq\exp{\left(-\frac{a_n^2}{n}\left(\frac{(x+\epsilon)^2}{2\sigma^2}+\delta'\right)\right)}.$$
Furthermore, by Markov's inequality,
$$v_n=\mathbb{P}(W_n\geq e^{a_n\epsilon})\leq e^{-a_n\epsilon}.$$
Hence,
$$0\leq\frac{v_n}{u_n}\leq\exp{\left(-a_n\epsilon+\frac{a_n^2}{n}\left(\frac{(x+\epsilon)^2}{2\sigma^2}+\delta'\right)\right)}\rightarrow0
\quad \text{as}\;n\rightarrow\infty,$$ since
$$\lim_{n\rightarrow\infty}\frac{-a_n\epsilon+\frac{a_n^2}{n}\left(\frac{(x+\epsilon)^2}{2\sigma^2}+\delta'\right)}{a_n}=-\epsilon<0.$$
Therefore,
$$\liminf_{n\rightarrow\infty}\frac{n}{a_n^2}\log \mathbb{P}\left(\frac{\log Z_n-n\mathbb{E}\log m_0}{a_n}\leq -x\right)
\geq\liminf_{n\rightarrow\infty}\frac{n}{a_n^2}\log
u_n=-\frac{(x+\epsilon)^2}{2\sigma^2}.
$$
Letting $\epsilon\rightarrow0$, we obtain (\ref{MDE4.3}). For
(\ref{MDE4.4}), the proof is similar. For every $\epsilon>0$,
\begin{eqnarray*}
&&\mathbb{P}\left( \frac{\log Z_n-n\mathbb{E}\log m_0}{a_n}\geq x\right)\\
&\leq&\mathbb{P}(W_n\geq e^{a_n\epsilon})+ \mathbb{P}\left(\frac{\log \Pi_n-n\mathbb{E}\log m_0}{a_n}\geq x-\epsilon\right)\\
&=&:v_n+\tilde{u}_n=\tilde{u}_n(1+v_n/\tilde{u}_n).
\end{eqnarray*}
Since $\lim_{n\rightarrow\infty}\frac{v_n}{\tilde{u}_n}=0$, we have
$$\limsup_{n\rightarrow\infty}\frac{n}{a_n^2}\log \mathbb{P}\left(\frac{\log Z_n-n\mathbb{E}\log m_0}{a_n}\geq x\right)
\leq\limsup_{n\rightarrow\infty}\frac{n}{a_n^2}\log\tilde{u}_n=-\frac{(x-\epsilon)^2}{2\sigma^2}.$$
Letting $\epsilon\rightarrow0$, we get (\ref{MDE4.4}).
\end{proof}

To prove Theorem \ref{MDT1.2}, we need to estimate the decay rate of the probabilities $\mathbb{P}(W_n\leq
e^{-a_n\epsilon})$ for $\epsilon >0$.

\begin{lem}
\label{MDL4.1}  If $\mathbb{E}W^{-s}<\infty$ for some $s>0$, then for any
positive
sequence $(a_n)$ satisfying $a_n\rightarrow\infty$, we have for all $
\epsilon>0$,
\begin{equation}  \label{MDE4.5}
\limsup_{n\rightarrow\infty}\frac{1}{a_n}\log \mathbb{P}(W_n\leq
e^{-a_n\epsilon})\leq-s\epsilon.
\end{equation}
\end{lem}

\begin{proof}[Proof]  By Markov's inequality and Lemma \ref{MDL3.1},
$$\mathbb{P}(W_n\leq e^{-a_n\epsilon})\leq\frac{\mathbb{E}W_n^{-s}}{e^{sa_n\epsilon}}\leq\frac{\mathbb{E}W^{-s}}{e^{sa_n\epsilon}}.$$
Thus
$$\frac{1}{a_n}\log \mathbb{P}(W_n\leq
e^{-a_n\epsilon})\leq\frac{1}{a_n}\log
\mathbb{E}W^{-s}-s\epsilon.$$ Taking the limit superior in the above inequality
gives (\ref{MDE4.5}).
\end{proof}

\begin{proof} [Another proof of Theorem \ref{MDT1.2}]
Lemma \ref{MDT4.1} gives one side of the desired results, so we
only need to prove the other side. By
Theorem \ref{MDT3.2}, there exists $s>0$ such that $\mathbb{E}W^{-s}<\infty$, so (\ref{MDE4.5}) holds for this $s$.
For $x>0$, we have for every $\epsilon>0$,
\begin{eqnarray*}
&&\mathbb{P}\left(\frac{\log Z_n-n\mathbb{E}\log m_0}{a_n}\leq -x \right)\\
&\leq&\mathbb{P}\left(W_n\leq e^{-a_n\epsilon}\right)+\mathbb{P}\left(\frac{\log \Pi_n-n\mathbb{E}\log m_0}{a_n}\leq -x+\epsilon\right)\\
&=&:v_n+u_n.
\end{eqnarray*}
By (\ref{MDE4.1}) and (\ref{MDE4.5}),
$\lim_{n\rightarrow\infty}\frac{v_n}{u_n}=0$, thus,
$$\limsup_{n\rightarrow\infty}\frac{n}{a_n^2}\log \mathbb{P}\left(\frac{\log Z_n-n\mathbb{E}\log m_0}{a_n}\leq -x\right)
\leq\limsup_{n\rightarrow\infty}\frac{n}{a_n^2}\log
u_n=-\frac{(x-\epsilon)^2}{2\sigma^2}.$$
 Letting $\epsilon\rightarrow0$, we obtain
\begin{equation}\label{MDE4.6}
\limsup_{n\rightarrow\infty}\frac{n}{a_n^2}\log \mathbb{P}\left(\frac{\log
Z_n-n\mathbb{E}\log m_0}{a_n}\leq -x\right) \leq-\frac{x^2}{2\sigma^2}.
\end{equation}
(\ref{MDE4.3}) and (\ref{MDE4.6}) yield (\ref{MDE1.3}). To prove
(\ref{MDE1.4}), on account of (\ref{MDE4.4}), it remains to show
that
\begin{equation}\label{MDE4.7}
\liminf_{n\rightarrow\infty}\frac{n}{a_n^2}\log \mathbb{P}\left(\frac{\log
Z_n-n\mathbb{E}\log m_0}{a_n}\geq x\right) \geq-\frac{x^2}{2\sigma^2}.
\end{equation}
Similarly, for every $\epsilon>0$,
\begin{eqnarray*}
&&\mathbb{P}\left(\frac{\log Z_n-n\mathbb{E}\log m_0}{a_n}\geq x \right)\\
&\geq& \mathbb{P}\left(\frac{\log \Pi_n-n\mathbb{E}\log m_0}{a_n}\geq x+\epsilon\right)-\mathbb{P}(W_n\leq e^{-a_n\epsilon})\\
 &=&:\tilde{u}_n-v_n.
\end{eqnarray*}
Again by (\ref{MDE4.1}) and (\ref{MDE4.5}),
$\lim_{n\rightarrow\infty}\frac{v_n}{\tilde{u}_n}=0$, thus,
$$\liminf_{n\rightarrow\infty}\frac{n}{a_n^2}\log \mathbb{P}\left(\frac{\log Z_n-n\mathbb{E}\log m_0}{a_n}\geq x\right)
\leq\liminf_{n\rightarrow\infty}\frac{n}{a_n^2}\log
\tilde{u}_n=-\frac{(x+\epsilon)^2}{2\sigma^2}.$$
 Letting $\epsilon\rightarrow0$, we obtain (\ref{MDE4.7}).
\end{proof}

We remark that, by Lemma \ref{MDL4.2} below, Theorem \ref{MDT1.2} is in fact equivalent to Theorem \ref{MDT4.3}.
So the direct proof of Theorem \ref{MDT1.2} leads to another proof of Theorem \ref{MDT4.3}.

\begin{lem}\label{MDL4.2}
Let $I$ be a continuous function on $\mathbb{R}$ satisfying\\
\begin{equation*}
	\begin{array}{l}
	\text{(a) $I(b)=\inf_{x\in\mathbb{R}}I(x)=0$ for some $b\in\mathbb{R}$;}\\
	\text{(b) $I$ is strictly increasing on $[b,\infty)$ and strictly decreasing on $(-\infty,b]$.}
	\end{array}
\end{equation*}
Let $(\mu_n)$ be a family of probability distribution on $\mathbb{R}$ and
let $(a_n)$ be a sequence of positive numbers satisfying $a_n\rightarrow\infty$.
Then the following statements (i) and (ii) are equivalent.\\
(i) For $x< b$,
$$\lim_{n\rightarrow\infty}\frac{1}{a_n}\log \mu_n((-\infty,x])=-I(x);$$
for $x> b$,
$$\lim_{n\rightarrow\infty}\frac{1}{a_n}\log \mu_n([x,+\infty))=-I(x).$$
(ii) $(\mu_n)$ satisfies a large deviation principle: for any measurable subset $B$ of $\mathbb{R}$,
\begin{eqnarray}
-\inf_{x\in B^o}I(x)&\leq&\liminf_{n\rightarrow\infty} \frac{1}{a_n}\log \mu_n(B)\label{MDE4.11}\\
				&\leq&\limsup_{n\rightarrow\infty} \frac{1}{a_n}\log \mu_n(B)\leq-\inf_{x\in \bar{B}}I(x),\label{MDE4.12}
\end{eqnarray}
where $B^o$ denotes the interior of $B$ and $\bar{B}$ its closure.
\end{lem}

This is a general result on large deviations. It shows that
 the large deviation principe holds if and only if  the corresponding limit exists  for  tail events,  when the rate function is continuous and  strictly monotone. This result would be known; as we have not found a reference, we shall give a proof in an
appendix by the end of the paper.
\\*


\section{Central limit theorems for  $W-W_n$  and $\log Z_n$}\label{MDS6}
In this section, we shall prove the results about central limit theorems.

We first prove the central limit theorem on   $W-W_n$ with exponential convergence rate,  using the results about the harmonic moments of $Z_n$ (i.e. Theorem \ref{MDT5.3} with $t<0$).

\begin{proof} [Proof of Theorem \ref{MDT6.1}]
Notice that
\begin{equation*}
 \Pi_n(W-W_n) = \sum_{i=1}^{Z_n}\left( W^{(n)}_{i}-1 \right),
\end{equation*}
where under $\mathbb{P}_{\xi }$, the random variables $W^{(n)}_{i} (i=1,2,...)$ are
independent of each other and independent of $Z_n$, and have common conditional distribution $
\mathbb{P}_{\xi }(W^{(n)}_{i}\in \cdot )=\mathbb{P}_{T^{n}\xi }(W\in \cdot )$. Notice that
if $a_0:=$ essinf $\frac{m_0(2)}{m_0^2}>1$, then $\delta^2_\infty\geq a_0-1>0$. Therefore the condition
$\mathbb{E}Z_1^{2+\epsilon}<\infty$ implies that $\mathbb{E}\left|\frac{W-1}{\delta_\infty}\right|^{2+\epsilon}<\infty$.
By the Berry-Esseen theorem (see \cite{chow}, Theorem 9.1.3),
for all $x\in\mathbb{R}$,
\begin{equation}\label{MDET5.3.2}
\left|\mathbb{P}_\xi\left(\frac{\Pi_n(W-W_n)}{\sqrt{Z_n}\delta_\infty(T^n\xi)}\leq x\right)-\Phi(x)\right|
\leq C_1\mathbb{E}_{T^n\xi}\left|\frac{W-1}{\delta_\infty}\right|^{2+\epsilon}\mathbb{E}_\xi Z_n^{-\epsilon/2},
\end{equation}
where $C_1$ is the Berry-Esseen constant.
Taking expectation in (\ref{MDET5.3.2}), we obtain for all $x\in\mathbb{R}$,
\begin{equation}\label{MDET5.3.3}
\left|\mathbb{P}\left(\frac{\Pi_n(W-W_n)}{\sqrt{Z_n}\delta_\infty(T^n\xi)}\leq x\right)-\Phi(x)\right|
\leq C_1\mathbb{E}\left|\frac{W-1}{\delta_\infty}\right|^{2+\epsilon}\mathbb{E} Z_n^{-\epsilon/2}.
\end{equation}
Since $\mathbb{E} p_1 < \mathbb{E} m_0^{-\epsilon/2}$, $\|p_1\|_\infty <1$ and (H) holds, the condition (iii) of Theorem  \ref{MDT5.3} is satisfied, so that
by Theorem \ref{MDT5.3}, there exists a constant $C_\epsilon>0$ such that
$$\lim_{n\rightarrow\infty}\frac{\mathbb{E }Z_n^{-\epsilon/2}}{\left(\mathbb{E}m_0^{-\epsilon/2}\right)^n}=C_\epsilon.$$
Combing this  with (\ref{MDET5.3.3}), we obtain (\ref{MDET5.3.1}).
\end{proof}

We then prove the central limit theorem on $\log Z_n$, using the central limit theorem on  $\log \Pi_n$.

\begin{proof}[Proof of Theorem \ref{MDT1.1}]
  Let $x\in\mathbb{R}$.
By the standard central limit theorem for i.i.d. random variables,
\begin{equation}\label{MDE2.1}
\lim_{n\rightarrow\infty}\mathbb{P}\left(\frac{\log \Pi_n-n\mathbb{E}\log
m_0}{\sqrt{n}\sigma}\leq x\right)=\Phi(x).
\end{equation}
By (\ref{MDE0}), we have for every $\epsilon>0$,
\begin{eqnarray}\label{MDE2.2}
&&\mathbb{P}\left(\frac{\log Z_n-n\mathbb{E}\log m_0}{\sqrt{n}\sigma}\leq
x \right)\nonumber\\&\leq&\mathbb{P}\left( \frac{\log
W_n}{\sqrt{n}}<-\epsilon\sigma \right)+\mathbb{P}\left(\frac{\log
\Pi_n-n\mathbb{E}\log m_0}{\sqrt{n}\sigma}\leq x+\epsilon\right).
\end{eqnarray}
Since $\lim_{n\rightarrow\infty}\frac{\log W_n}{\sqrt{n}}=0$ a.s.,  we have
\begin{equation}\label{MDE2.3}
\lim_{n\rightarrow\infty}\mathbb{P}\left( \frac{\log
W_n}{\sqrt{n}}<-\epsilon\sigma \right)=0.
\end{equation}
Taking the  limit  superior in (\ref{MDE2.2}), and applying (\ref{MDE2.1}) and (\ref{MDE2.3}), we obtain
$$\limsup_{n\rightarrow\infty}\mathbb{P}\left( \frac{\log Z_n-n\mathbb{E}\log m_0}{\sqrt{n}\sigma}\leq x \right)\leq\Phi(x+\epsilon).$$
Letting $\epsilon\rightarrow0$, we get the upper bound. For the lower
bound, observe that
\begin{eqnarray}\label{MDE2.4}
&&\mathbb{P}\left( \frac{\log Z_n-n\mathbb{E}\log m_0}{\sqrt{n}\sigma}\leq x \right)\nonumber\\
&\geq&\mathbb{P}\left(\frac{\log \Pi_n-n\mathbb{E}\log m_0}{\sqrt{n}\sigma}\leq
x-\epsilon\right)-\mathbb{P}\left( \frac{\log
W_n}{\sqrt{n}}>\epsilon\sigma \right).
\end{eqnarray}
Similarly,
\begin{equation*}
\lim_{n\rightarrow\infty}\mathbb{P}\left(\frac{\log
W_n}{\sqrt{n}}>\epsilon\sigma\right)=0.
\end{equation*}
Taking the limit inferior in (\ref{MDE2.4}) and  letting
$\epsilon\rightarrow0$,  we get
$$\liminf_{n\rightarrow\infty}\mathbb{P}\left( \frac{\log \Pi_n-n\mathbb{E}\log m_0}{\sqrt{n}\sigma}\leq x \right)\geq\Phi(x).$$
So (\ref{MDE1.2}) is proved.
\end{proof}


\section{Appendix: proof of  Lemma \ref{MDL4.2}}

\begin{proof}[Proof of  Lemma \ref{MDL4.2}]
It is clear that (ii) implies (i) since $I$ is continuous. We need to prove (i) implies (ii). Firstly, we show (\ref{MDE4.11}). For $x\in B^o$,
consider the case where $x\geq b$. Then $B^o$ contains an interval $[x+\epsilon_1,x+\epsilon_2)$ for some $0<\epsilon_1<\epsilon_2$. Consequently, by (i),
$\forall \epsilon>0$, there exists $n_\epsilon>0$  such that $\forall n\geq n_\epsilon$,
\begin{eqnarray*}
\mu_n(B)&\geq &\mu_n([x+\epsilon_1,x+\epsilon_2))\\
&=&\mu_n([x+\epsilon_1,\infty))-\mu_n([x+\epsilon_2,\infty))\\
&\geq& e^{-a_n(I(x+\epsilon_1)+\epsilon)}-e^{-a_n(I(x+\epsilon_2)-\epsilon)}.
\end{eqnarray*}
Since $I$ is strictly increasing on $[b,\infty)$, we can take $\epsilon>0$ small enough such that
$I(x+\epsilon_1)+\epsilon<I(x+\epsilon_2)-\epsilon$. Therefore,
$$\liminf_{n\rightarrow\infty} \frac{1}{a_n}\log \mu_n(B)\geq- I(x+\epsilon_1)-\epsilon.$$
Letting $\epsilon,\epsilon_1\rightarrow 0$, we get
\begin{equation}\label{MDE4.13}
\liminf_{n\rightarrow\infty} \frac{1}{a_n}\log \mu_n(B)\geq- I(x)
\end{equation}
If $x<b$,   we obtain(\ref{MDE4.13}) by a similar argument. So (\ref{MDE4.13}) holds for all $x\in B^o$, which yields (\ref{MDE4.11}).

Now we show (\ref{MDE4.12}). If $b\in \bar{B}$, then (\ref{MDE4.12}) is obvious since $\mu_n(B)\leq 1$ and the right side of (\ref{MDE4.12})
is $0$. Assume that $b\notin \bar{B}$. Let $B_1=B\bigcap(-\infty, b]$ and $B_2=B\bigcap(b,\infty)$ so that $B=B_1\bigcup B_2$. Then
$$B_1\subset (-\infty, b_1]\;(\text {if $B_1\neq \emptyset $})\qquad \text{and}\qquad
B_2\subset [b_2,\infty)\;(\text {if $B_2\neq \emptyset $}),$$
where $b_1:=\sup B_1$ and $b_2:=\inf B_2$. Assume that $B_1\neq \emptyset $ and $B_2\neq \emptyset $. As $b\notin \bar{B}$, we
have $b_1<b<b_2$. By (i), $\forall \epsilon>0$, there exists $n_\epsilon>0$  such that $\forall n\geq n_\epsilon$,
\begin{eqnarray*}
\mu_n(B)&\leq & \mu_n([-\infty,b_1])+\mu_n([b_2,\infty))\\
&\leq& e^{-a_n(I(b_1)-\epsilon)}+e^{-a_n(I(b_2)-\epsilon)}\\
&\leq&2e^{-a_n(I_0-\epsilon)},
\end{eqnarray*}
where $I_0:=\min\{I(b_1),I(b_2)\}=\inf_{x\in\bar{B}}I(x)$. Therefore,
$$\limsup_{n\rightarrow\infty} \frac{1}{a_n}\log \mu_n(B)\leq- I_0+\epsilon.$$
Letting $\epsilon\rightarrow 0$, we obtain
\begin{equation*}
\limsup_{n\rightarrow\infty} \frac{1}{a_n}\log \mu_n(B)\leq- I_0=-\inf_{x\in\bar{B}}I(x).
\end{equation*}
If $B_1=\emptyset$ or $B_2=\emptyset$, we obtain (\ref{MDE4.12}) by a similar argument.
\end{proof}

\bigskip
{\emph Acknowledgement.} The authors would like to thank an anonymous referee for valuable comments and remarks.


\begin{thebibliography}{99}

\bibitem{af1} V.I. Afanasyev, J. Geiger, G. Kersting, V.A. Vatutin. Criticality for branching processes in random environment.  Ann. Probab.  33  (2005),  no. 2, 645-673.

\bibitem{af2} V.I. Afanasyev, J. Geiger, G. Kersting, V.A. Vatutin. Functional limit theorems for strongly subcritical branching processes in random environment.  Stochastic Process. Appl.  115  (2005),  no. 10, 1658-1676.


\bibitem{a}K.B. Athreya,  P.E. Ney. Branching
Processes. Springer, Berlin, 1972.

\bibitem{a1}K.B. Athreya,  S. Karlin. On branching processes in random environments I, II. Ann. Math. Statist. 42 (1971), 1499-1520, 1843-1858.

\bibitem{ba}V. Bansaye,  J. Berestycki. Large deviations for branching
processes in random environment. Markov Process. Related Fields 15 (2009), 493-524.

\bibitem{ba2}V. Bansaye,   C. B\"oinghoff. Upper large deviations for branching
processes in random environment with heavy tails. Preprint (2010).

\bibitem{bd} C. B\"oinghoff, E.E. Dyakonova,  G. Kersting, V.A. Vatutin. Branching processes in random environment which extinct at a given moment.  Markov Process. Related Fields  16  (2010),  no. 2, 329-350.

\bibitem{bo}C. B\"oinghoff,  G. Kersting. Upper large deviations of branching processes in a random environment--Offspring distributions with geometrically bounded tails.  Stoch. Proc. Appl. 120 (2010), 2064-2077.


\bibitem{chow}Y.S. Chow, H. Teicher. Probability theory: Independence, Interchangeability and Martingales. Springer-Verlag, New York, 1988.

\bibitem{z}A. Dembo, O. Zeitouni. Large deviations Techniques and
Applications. Springer, New York, 1998.

\bibitem{liu1}Y. Guivarc'h, Q. Liu. Propri\'et\'es asymptotiques des
processus de branchement en environnement
al\'eatoire.  C. R. Acad. Sci. Paris, Ser I. 332 (2001), 339-344.

\bibitem{ham}B. Hambly. On the limit distribution of a supercritical
branching process in a random
environment. J. Appl. Prob. 29 (1992), 499-518.

\bibitem{heyde1}C.C. Heyde.  Some central limit analogues for super-critical Galton-Watson process. J. Appl. Probab. 8 (1971), 52-59.

\bibitem{heyde2} C.C. Heyde,  B.M. Brown. An invariance principle and some convergence rate results for branching processes. Z. Wahrscheinlichkeitstheorie verw. Geb. 20 (1971), 189-192.

\bibitem{jajer} P. Jagers.  Galton-Watson processes in varying environments. J. Appl. Prob. 11 (1974), 174-178.

\bibitem{kozlov}M.V.  Kozlov. On large deviations of branching processes in a random environment : geometric distribution of descendants. Discrete Math. Appl. 16 (2006) 155-174.

\bibitem{liu1999}Q. Liu. Asymptotic properties of supercritical age-dependent branching
processes and homogeneous branching random walks.  Stoch. Proc. Appl.
82 (1999), 61-87.

\bibitem{liu2}Q. Liu. Asymptotic properties and absolute continuity
of laws stable by random weighted mean.
Stoch. Proc. Appl. 95 (2001), 83-107.

\bibitem{liu3}Q. Liu. Local dimensions of the branching measure on a Galton-Watson tree. Ann. Inst. Henri. Poincar\'e, Probabiliti\'es et Statistique  37 (2001), 195-222.

\bibitem{liu4}Q. Liu,  A. Rouault. Limit theorems for Mandelbrot's multiplicative cascades. Ann. Appl. Proba. 10 (2000), 218-239.

\bibitem{liu5}Q. Liu. The growth of an entire characteristic function and the tail probabilities of the limit of a tree martingale. In: Chauvin B., Cohen S.,  Rouault A., Trees. Progress in Probability, vol.40, Birkh\"auser, Basel, (1996), 51-80.

\bibitem{ney}P.E. Ney,  A.N. Vidyashanker. Harmonic moments and large deviation rates for supercritical branching process. Ann. Appl. Proba. 13 (2003), 475-489.

\bibitem{smith}W.L. Smith, W.E. Wilkinson.  On branching processes in random environments. Ann. Math. Statist. 40 (1969), 814-827.

\bibitem{tanny1}D. Tanny. Limit theorems for branching processes in
a random environment. Ann. Proba. 5 (1977), 100-116.

\bibitem{tanny2}D. Tanny. A necessary and sufficient condition for a
branching process in a random environment to grow like the product
of its means. Stoch. Proc. Appl. 28 (1988), 123-139.

\bibitem{wang}H. Wang, Z. Gao, Q. Liu. Central limit theorems for a branching process in a random environment.  Stat. Prob. Letters 81 (2011) 539-547.


\end{thebibliography}
\end{document}